\documentclass[11pt]{article}
\usepackage[utf8]{inputenc}
\usepackage[T1]{fontenc}
\usepackage{fullpage}
\usepackage{amsmath}
\usepackage{amsthm}
\usepackage{amssymb}
\usepackage{enumerate}
\usepackage{upref}
\RequirePackage{doi}
\usepackage{hyperref}
\hypersetup{colorlinks=true,
            citecolor=green!40!black,
            linkcolor=red!20!black,
            urlcolor=blue!40!black,
            filecolor=cyan!30!black} 

\usepackage{mathtools}
\usepackage{color}
\usepackage{xcolor}
\usepackage{thmtools}
\usepackage{thm-restate}
\usepackage{authblk}
\usepackage[capitalise]{cleveref}

\numberwithin{equation}{section}

\declaretheorem[name=Theorem,numberwithin=section]{thm}

\theoremstyle{definition}
\newtheorem{cor}[equation]{Corollary}
\newtheorem{lem}[equation]{Lemma}
\newtheorem{prop}[equation]{Proposition}
\newtheorem{prob}{Problem}
\newtheorem{obs}[equation]{Observation}
\newtheorem{eg}[equation]{Example}

\crefname{prob}{Problem}{Problems}
\Crefname{prob}{Problem}{Problems}
\crefname{eg}{Example}{Examples}
\Crefname{eg}{Example}{Examples}
\crefname{thm}{Theorem}{Theorems}
\Crefname{thm}{Theorem}{Theorems}

\crefformat{prop}{Proposition~#2\textup{#1}#3}
\Crefname{prop}{Proposition}{Propositions}
\crefname{prop}{proposition}{propositions}
\crefmultiformat{prop}{Propositions~#2\textup{#1}#3}{ and~#2\textup{#1}#3}{, #2\textup{#1}#3}{ and~#2\textup{#1}#3}

\DeclareMathOperator{\Aut}{Aut}

\newcommand{\pref}[1]{\textup(\ref{#1}\textup)}

\newcommand{\fullcref}[2]{\cref{#1}\pref{#1-#2}}

\title{Paths closing to cycles in symmetric graphs}
\title{Closing paths to cycles in symmetric graphs}

\author[1]{Martin Milani\v c}

\author[2]{\DJ or\dj e Mitrovi\' c}

\affil[1]{FAMNIT and IAM, University of Primorska, Koper, Slovenia}

\affil[2]{Department of Mathematics, University of Auckland, Auckland, New Zealand}

\date{}

\begin{document}

\maketitle

\begin{abstract}
It was shown by Beisegel, Chudnovsky, Gurvich, Milanič, and Servatius in 2022 that every induced $2$-edge path in a vertex-transitive graph closes to an induced cycle.
Similar results were obtained for 3-edge paths closing to cycles in edge-transitive graphs, where the cycle can be assumed to be induced if the path is induced. 
Motivated by these results, we consider the following problem: 
For a given class of graphs, determine all integers $\ell\geq 0$ such that for every graph in the class, every path of length at most $\ell$ closes to a cycle. 
We also consider the variant of the problem for induced paths closing to induced cycles.
We completely solve these problems for the classes of (finite) vertex-transitive graphs, edge-transitive graphs, and edge-transitive graphs that are not stars.
For all but one case of a negative answer, we provide infinite families of connected counterexamples.
\end{abstract}

\section{Introduction}

All graphs considered in this paper are finite, simple, and undirected.
 An automorphism of a graph is a permutation of its vertices mapping edges to edges and non-edges to non-edges. 
 The relationship between existence of symmetry (i.e., a nontrivial automorphism) in a graph and its structural properties is one of the main objects of study in algebraic graph theory (see, e.g., \cite{godsil2001algebraic}).
 The general principle is that the existence of symmetry in a graph restricts its structural properties; a trivial example of this is that vertex-transitive graphs are regular, that is, all vertices have the same degree.
    
These structural restrictions also have to do with the existence and relations between subgraphs of the graph in question. 
In this article, the subgraphs we will be particularly interested in are paths and cycles.
There is a number of results on these particular substructures in symmetric graphs already available in the literature.
For example, in {\cite[Theorem~III]{Tutte1947_FamilyOfCubicGraphs}} Tutte shows that the length of a shortest cycle of an $s$-arc-transitive graph with degree at least $3$ is bounded below by $2s-2$ (see also {\cite[Chapter 4]{godsil2001algebraic}}). 
Moreover, if equality is attained, the graph must be bipartite and have diameter $s-1$.

The starting point of our investigation is the $2022$ work \cite{BCGMS2022} by Beisegel, Chudnovsky, Gurvich, Milanič, and Servatius. 
The authors considered different generalisations of the concept of a simplicial vertex (i.e., a vertex in a graph whose neighbourhood forms a clique), including but not limited to \textsl{avoidable vertices}, \textsl{avoidable edges}, \textsl{simplicial paths}, and \textsl{avoidable paths} (we refer the reader to~\cite{BCGMS2022} for formal definitions of these concepts). 
Various existence results for these substructures in general graphs have been obtained (see also~\cite{MR4245221,Gurvich2021Shifting,MR4904889} for generalisations). 
In the final section of their paper, the authors considered the implications of these results for graphs with certain symmetry properties, in particular for vertex-transitive and edge-transitive graphs.

\begin{sloppypar}
For example, the fact that every graph with an edge contains a \textsl{pseudo-avoidable edge} \hbox{({\cite[Corollary 3.2]{BCGMS2022}})} proves the following for edge-transitive graphs.
\end{sloppypar}

\begin{restatable}[{\cite[Theorem~6.2]{BCGMS2022}}]{thm}{ETBCGMS}
\label{ET-3Paths}
Every $3$-edge path in an edge-transitive graph closes to a cycle.
\end{restatable}

Similarly, the fact that every graph containing an edge contains an avoidable vertex ({\cite[Theorem~1.4]{BCGMS2022}}) as well as an avoidable edge ({\cite[Theorem~5.4]{BCGMS2022}}) implies the following results for vertex-transitive and edge-transitive graphs.

\begin{thm}[{\cite[Theorem~6.1]{BCGMS2022}}]\label{BCGMS_VTInd2paths}
Every induced $2$-edge path in a vertex-transitive graph closes to an induced cycle.
\end{thm}

\begin{restatable}[{\cite[Theorem~6.3]{BCGMS2022}}]{thm}{ETindBCGMS}
\label{IndPtoC_ET3-paths}
Every induced $3$-edge path in an edge-transitive graph closes to an induced cycle.
\end{restatable}

Related questions, for not necessarily symmetric graphs, were studied by Dirac and Thomassen (see~\cite{MR335352}).
For all $\ell\ge 2$, they characterised graphs containing a path of length $\ell$ such that every path of length $\ell$ closes to a cycle of length $\ell+1$.
They showed that up to isomorphism, for every $\ell$ there are at most three such graphs with $\ell+1$ vertices: the complete graph $K_{\ell+1}$, the $(\ell+1)$-vertex cycle, and the balanced complete bipartite graph with parts of size $(\ell+1)/2$ for odd $\ell$.
Furthermore, they showed that in a connected graph with at least three vertices, every path closes to a cycle if and only if the graph is a complete graph, a cycle graph, or a balanced complete bipartite graph.
These are also precisely the graphs in which every path closes to a Hamiltonian cycle (see Chartrand and Kronk~\cite{MR234852}).
Further related results were obtained by Parsons~\cite{MR505821}, Reid and Thomassen~\cite{MR429637}, and Thomassen~\cite{MR349484}.

Motivated by these results, we consider the following questions.

\begin{prob}[PtoC Problem]\label{PtoCProb}
    Given a graph class $\mathcal{G}$, determine all integers $\ell \geq 0$ with the property that for $G\in \mathcal{G}$, every $\ell$-edge path in $G$ closes to a cycle.
\end{prob}

\begin{prob}[Induced PtoC Problem]\label{IndPtoCProb}
    Given a graph class $\mathcal{G}$, determine all integers $\ell \geq 0$ with the property that for $G\in \mathcal{G}$, every induced $\ell$-edge path in $G$ closes to an induced cycle.
\end{prob}

Note that when studying these questions, it is indeed natural to consider connected graphs only, as both paths and cycles are connected subgraphs, and hence, each is contained in exactly one connected component.

We observe that both \cref{PtoCProb,IndPtoCProb} can easily become uninteresting depending on the class~$\mathcal{G}$. 
For example, the class of paths shows that there are no nonnegative integers $\ell$ such that for every sufficiently large connected graph $G$, every $\ell$-edge path in $G$ closes to a cycle.
Of course, these examples are in a sense trivial since they are not $2$-connected. 
However, for $\ell\ge 2$, replacing each internal vertex of an $\ell$-edge path with two adjacent vertices results in a \hbox{$2$-connected} graph that contains an induced $\ell$-edge path that does not close to an induced cycle. 
For $\ell\ge 3$, the graph obtained from the complete graph with $\ell+1$ vertices by deleting one edge is a \hbox{$2$-connected} graph that contains an $\ell$-edge path that does not close to a cycle.

Notably, the above examples are neither vertex- nor edge-transitive. 
In this paper, we study \cref{PtoCProb,IndPtoCProb} for both vertex-transitive and edge-transitive graphs, expanding on the work already done in \cite{BCGMS2022}.

In particular, we prove the following results, which completely resolve \cref{PtoCProb,IndPtoCProb} for the class of connected vertex-transitive graphs, generalising \cref{BCGMS_VTInd2paths}.

\begin{restatable}[PtoC for VT graphs]{thm}{StatePtoCVT}
\label{PtoCVT}
Let $G$ be a connected vertex-transitive graph with $|V(G)|\geq 3$. Then, every path of length at most $4$ in $G$ closes to a cycle.
Moreover, for every integer $\ell \geq 5$ there exist infinitely many connected vertex-transitive graphs each containing a path of length $\ell$ that does not close to any cycle.
\end{restatable}

\begin{restatable}[Induced PtoC for VT graphs]{thm}{StateIndPtoCVT}
\label{IndPtoCVT}
    Let $G$ be a connected vertex-transitive graph with $|V(G)|\geq 3$. 
    Then, every induced path of length at most $2$ in $G$ closes to an induced cycle. 
    Moreover, for every integer $\ell \geq 3$ there exist infinitely many connected vertex-transitive graphs each containing an induced path of length $\ell$ that does not close to an induced cycle.
\end{restatable}

As will be discussed in \cref{Section:PtoCET}, \cref{ET-3Paths} already provides a complete solution to \cref{PtoCProb} for the class of connected edge-transitive graphs. 
However, a slightly different result is obtained when one excludes star graphs from this class.

\begin{restatable}[PtoC for non-star ET graphs]{thm}{StatePtoCETnoStars}
\label{PtoCETnoStars}
    Let $G$ be a connected edge-transitive graph that is not isomorphic to a star. Then every path in $G$ of length at most $3$ closes to a cycle. 
    Moreover, for every integer $\ell \geq 4$, there exist infinitely many connected edge-transitive graphs each containing a path of length $\ell$ that does not close to a cycle.
\end{restatable}

Similarly, \cref{IndPtoC_ET3-paths} already solves \cref{IndPtoCProb} for connected edge-transitive graphs, as will be discussed in \cref{Section:IndPtoCET}. 
It turns out that star graphs are yet again the sole obstacle to a more general result.

\begin{restatable}[Induced PtoC for non-star ET graphs]{thm}{StateIntPtoCETnoStars}
\label{IntPtoCETnoStars}
    Let $G$ be a connected edge-transitive graph that is not a star. 
    Then, every induced path of length at most $3$ in $G$ closes to an induced cycle. 
    Moreover, for every integer $\ell\geq 4$, there exists a connected edge-transitive graph containing an induced path of length $\ell$ not contained in any induced cycle.
\end{restatable}

Our results are summarised in following table.

\begin{table}[!htbp]
    \centering
    \begin{tabular}{c|c|c}
         $G$ & PtoC Problem &  Induced PtoC problem \\
\hline

         Vertex-transitive with $|V(G)|\geq 3$ &  0 $\leq \ell \leq 4$ & $0\leq \ell \leq 2$ \\
         Edge-transitive & $\ell = 3$ & $\ell = 3$ \\
         Edge-transitive and not a star &  $0\leq \ell \leq 3$ & $0\leq \ell \leq 3$
    \end{tabular}
    \caption{Summary of solutions to \cref{PtoCProb,IndPtoCProb} for the classes of vertex-transitive and edge-transitive graphs}
    \label{Table:Summary}
\end{table}

The paper is structured as follows. 
In \cref{Section:Preliminaries}, we recall some standard notions and results used throughout the paper. 
In \cref{Section:PtoCVT,Section:PtoCET}, we consider \cref{PtoCProb} for vertex-transitive and edge-transitive graphs, respectively, while in \cref{Section:IndPtoCVT,Section:IndPtoCET}, we consider \cref{IndPtoCProb} for vertex-transitive and edge-transitive graphs, respectively.
We conclude in \Cref{sec:open} with some open problems.

\section{Preliminaries} \label{Section:Preliminaries}

A graph $G$ is an ordered pair $(V,E)$, where $V = V(G)$ is the vertex-set of $G$ and $E = E(G)$ is the edge-set of $G$.
Given a graph $G$, two vertices $x,y\in V(G)$ are said to be \emph{adjacent}  if $\{x,y\}\in E(G)$. 
Vertices $y\in V(G)$ such that $\{x,y\}\in E(G)$ are called \emph{neighbours} of $x$ in $G$. 
The number of neighbours of $x$ is the \emph{degree} of $x$, denoted by $d(x)$.
The \emph{minimum degree} of a graph $G$ is defined as $\delta(G)\coloneqq \min\{d(x)\colon x\in V(G)\}$.
A graph $G$ is said to be \emph{regular} if all the vertices have the same degree.
If this common degree is $d$, the graph is said be \emph{$d$-regular}.
We denote the $n$-vertex complete graph by $K_n$.
A \emph{graph class} is a collection of graphs closed under isomorphism.

A graph $G$ is said to be \emph{bipartite} if it has a \emph{bipartition}, that is, a partition of $V(G)$ into two sets of pairwise nonadjacent vertices. A \emph{complete bipartite graph} is a bipartite graph with a bipartition $\{V_1,V_2\}$ such that every two vertices $u\in V_1$ and $v\in V_2$ are adjacent (note that there are no other edges).
For two nonnegative integers $m$ and $n$, we denote by $K_{m,n}$ the complete bipartite graph with a bipartition $\{V_1,V_2\}$ such that $|V_1| = m$ and $|V_2|= n$.
For any $n\ge 0$, the complete bipartite graph $K_{1,n}$ is said to be a \emph{star graph} (or simply a \emph{star}).

A \emph{path} $P$ in a graph $G = (V,E)$ is a sequence of pairwise distinct vertices $P=(x_0,x_1,\ldots,x_\ell)$ such that $\{x_i,x_{i+1}\}\in E$ for all $0\leq i\leq \ell-1$. If $P$ contains $\ell$ edges (or equivalently, $\ell+1$ vertices) we call it an \emph{$\ell$-edge path} or a \emph{path of length $\ell$}.
A \emph{cycle} $C$ in $G$ is a sequence of vertices $(x_0,\ldots,x_{\ell-1},x_\ell = x_0)$ with $\ell\ge 3$ such that 
$x_i\neq x_j$ for $0\le i<j\le \ell-1$ and 
$\{x_i,x_{i+1}\}\in E$ for $0\leq i\leq \ell-1$.
Such a cycle $C$ is also referred to as an \emph{$\ell$-cycle} or a \emph{cycle of length $\ell$}.
It will often be convenient to identify a path in a graph $G$ with the subgraph formed by the vertices and the edges of the path, and the same for cycles.
If this subgraph is an induced subgraph of $G$,
we say that the path, respectively cycle, is \emph{induced}. 
We say that a path $P$ in a graph $G$ \emph{closes to a cycle} in $G$ if there exists a cycle $C$ in $G$ such that $P$ is a subgraph of $C$.
Similarly, we say that an induced path $P$ \emph{closes to an induced cycle in $G$} if there exists an induced cycle $C$ such that $P$ is an induced subgraph of $C$.

Given an integer $k\ge 0$, a graph $G$ is said to be \emph{$k$-connected} if $|V|\ge k+1$ and for each set $S\subseteq V(G)$ such that $|S| < k$, the graph $G-S$ is connected.
By $\kappa(G)$ we denote the \emph{(vertex) connectivity} of $G$, that is, the maximum integer $k\ge 0$ such that $G$ is $k$-connected.

\begin{obs}
\label{kappaatmostk}
Let $G$ be a $d$-regular graph and let $v\in V(G)$.
Then $v$ is an isolated vertex in the graph obtained by removing all $d$ neighbours of $v$ in $G$.
In particular, we conclude that $\kappa(G)\leq d$.
\end{obs}

We next recall basic concepts related to symmetries of graphs.

Let $G=(V,E)$ be a graph.
A permutation of vertices $\phi\colon V\rightarrow V$ is called an \emph{automorphism} of $G$ if for all distinct vertices $x,y\in V$, it holds that $\{x,y\}\in E$ if and only if $\{\phi(x),\phi(y)\}\in E$.
The set of all automorphisms of $G$ forms a group under function composition and is be denoted by $\Aut(G)$.
We refer to $\Aut(G)$ as the \emph{automorphism group} of $G$. A graph $G=(V,E)$ is said to be \emph{vertex-transitive} if for every two vertices $x,y\in V$ there exists an automorphism $\phi\in \Aut(G)$ such that $\phi(x) = y$.

\begin{obs}\label{VTimpliesRegular}
Graph automorphisms preserve vertex degrees. In particular, vertex-transitive graphs are regular.
\end{obs}

Let $G = (V,E)$ be a graph and $\phi\in \Aut(G)$ an automorphism of $G$.
The automorphism $\phi$ maps vertices to vertices, and it naturally induces a mapping from edges to edges, $\phi\colon E\to E$, simply by setting $\phi(\{x,y\}) =  \{\phi(x), \phi(y)\}$ for all $x,y\in V$ such that $x$ is adjacent to $y$ in $G$.
It is not difficult to verify that the induced mapping is a permutation of $E$.
A graph $G=(V,E)$ is said to be \emph{edge-transitive} if for every two edges $e,f\in E$ there exists an automorphism $\phi\in \Aut(G)$ such that $\phi(e) = f$. The \emph{line graph} of a graph $G$, denoted $L(G)$, is a graph with vertex set $E(G)$ with two distinct edges $e,f\in E(G)$ being adjacent in $L(G)$ if and only if they share a vertex in $G$.

\begin{obs}
\label{AutLine}
For $\phi\in\Aut(G)$, permutation of $E(G)$ induced by $\phi$ is an automorphism of $L(G)$. In particular, the line graph of an edge-transitive graph is vertex-transitive.
\end{obs}

\begin{thm}[{\cite[Lemma 3.2.1]{godsil2001algebraic}}]
\label{edgetran-notvertrans}
Every edge-transitive graph with minimum degree at least $2$ is vertex-transitive or  bipartite. 
\end{thm}

The following lower bound on the connectivity of a vertex-transitive graph was proved independently by Mader~\cite{MR289343} and Watkins~\cite{MR266804}.

\begin{thm}[{\cite[Theorem~3.4.2]{godsil2001algebraic}}]
\label{VTKappaBound}
Let $G$ be a connected $d$-regular vertex-transitive graph.
Then $\kappa(G)\ge \frac{2}{3}(d+1)$.
\end{thm}

For edge-transitive graphs, the connectivity can be expressed {\it exactly} in terms of vertex degrees, as follows.

\begin{thm}[Watkins~\cite{MR266804}]
\label{ETKappaBound}
Let $G$ be a connected edge-transitive graph.
Then $\kappa(G)= \delta(G)$.
\end{thm}

The following results, although simple observations, will be some of our main tools when proving results and constructing counterexamples in subsequent sections.

\begin{lem}
\label{PathToCycle_KappaTrick}
Let $G$ be a connected graph with minimum degree at least $2$. 
Then, every path in $G$ of length at most $\kappa(G)$ closes to a cycle.
\end{lem}

\begin{proof}
Let $P=(v_0,v_1,\ldots,v_{\ell-1},v_\ell)$ be an $\ell$-edge path in $G$ with $2\leq \ell\leq \kappa(G)$. If $v_0 = v_\ell$, then $P$ is already a cycle, so we can assume $v_0\neq v_\ell$. As $\kappa(G)> \ell-1\geq 1$, the graph $G^*\coloneqq G\setminus \{v_1,\ldots,v_{\ell-1}\}$ obtained by deleting the $\ell-1$ vertices $v_1,\ldots,v_{\ell-1}$ from $G$ is connected. 
Hence, there exists a path $P^*=(v_0=w_0,w_1,\ldots,w_k=v_\ell)$ in $G^*$ connecting $v_0$ and $v_\ell$. 
Note that $P^*$ is also a path in $G$. Furthermore, $P$ and $P^*$ share no edges and the only vertices they have in common are $v_0$ and $v_\ell$. 
Since $k\ge 1$ and $\ell\ge 2$, we infer that $P$ and $P^*$ form a cycle $(v_0=w_0,\ldots,w_k = v_\ell,v_{\ell-1},\ldots,v_1,v_0)$ of length $k+\ell\ge 3$ in $G$.
By construction, this cycle contains~$P$.

Since $\delta(G)\ge 2$, every $0$-edge path (i.e., vertex) closes to a $1$-edge path (i.e., edge), which in turn closes to a $2$-edge path.
Hence, by the first part of the proof, every $\ell$-edge path in $G$ with $\ell\in \{0,1\}$ also closes to a cycle.
\end{proof}

\begin{obs}
\label{PathNotInCycleStrategy}
A useful strategy for constructing examples of paths that are not contained in any cycle is the following. 
Choose a vertex $x\in V(G)$. 
Suppose that it is possible to construct a path $P$ starting at $x$ that visits all the neighbours of $x$ in some order, possibly with extra vertices in between, and terminates at a non-neighbour of $x$. 
Then $P$ does not close to any cycle in $G$, as the starting vertex $x$ cannot be reached anymore without revisiting one of its neighbours.
\end{obs}

\section{Paths closing to cycles in vertex-transitive graphs}\label{Section:PtoCVT}

In this section, we study \cref{PtoCProb} for connected vertex-transitive graphs. We start by proving the following result for the more general class of regular graphs.

\begin{prop}\label{RegPtoC}
    Let $G$ be a connected $d$-regular graph such that
    \begin{enumerate}[(1)]
        \item \label{RegPtoC-kappa>=4} $\kappa(G)\geq 4$, or
        \item \label{RegPtoC-cubic}$\kappa(G) = d = 3$.
    \end{enumerate}
Then, every path in $G$ of length at most $4$ closes to a cycle.
\end{prop}

    \begin{proof}
    Suppose first that $\kappa(G)\ge 4$, that is, we are in case \pref{RegPtoC-kappa>=4}.
    Then, \cref{kappaatmostk} implies $k \geq \kappa(G)\geq 4$, and the conclusion follows immediately by \cref{PathToCycle_KappaTrick}.

    Suppose next that $G$ is a connected $3$-regular graph with $\kappa(G)=3$, that is, we are in case \pref{RegPtoC-cubic}.
    Let $P$ be an $\ell$-edge path in $G$ with $0\leq \ell\leq 4$.
    If $\ell \le 3$, then we can again apply \cref{PathToCycle_KappaTrick}.
    Suppose now that $\ell = 4$ and let $P = (v_0,v_1,v_2,v_3,v_4)$.
    Let $G' = G\setminus \{v_1,v_3\}$ be a graph obtained from $G$ by removing vertices $v_1$ and $v_3$. 
     Since $\kappa(G)= 3$, the graph $G'$ is connected. 
     In particular, it contains a path $P'$ from $v_0$ to $v_4$.
     Note that vertex $v_2$ has degree $1$ in $G'$ and therefore $v_2$ cannot belong to $P'$.
     It follows that $P$ and $P'$ form a cycle in $G$ containing $P$.
     This completes the proof.
     \end{proof}

Building on \cref{RegPtoC}, we can obtain the following result for vertex-transitive graphs.

\begin{prop}\label{VTPtoC_l<=4}
    Let $G$ be a connected vertex-transitive graph with $|V(G)|\geq 3$. 
    Then every path in $G$ of length at most $4$ closes to a cycle.
\end{prop}

\begin{proof}
    Since $G$ is vertex-transitive, it is also $d$-regular for some integer $d\ge 0$. 
    As $|V(G)|\geq 3$, it follows that $G$ is not isomorphic to $K_1$ nor $K_2$, and hence, $d\geq 2$. 
    If $d = 2$, then $G$ is isomorphic to a cycle graph $C_n$ for some $n\ge 3$, for which the claim clearly holds.

    If $d = 3$, then \cref{VTKappaBound} shows that $\kappa(G)\geq \frac{2}{3}(3 + 1) = \frac{8}{3} > 2$. 
    Since $\kappa(G)$ is an integer, it follows that $\kappa(G) \geq 3$. By \cref{kappaatmostk}, we also have that $\kappa(G)\leq d = 3$. 
    In particular, $\kappa(G) = 3$ and the conclusion follows by \fullcref{RegPtoC}{cubic}.

    If $d \geq 4$, by another application of \cref{VTKappaBound}, we have that $\kappa(G)\geq \frac{2}{3}(d+1) \geq \frac{2}{3}(4+1)=\frac{10}{3} > 3$. 
    Hence, $\kappa(G)\geq 4$ and the conclusion follows by \fullcref{RegPtoC}{kappa>=4}. 
\end{proof}

The bound on the length of the path in \cref{RegPtoC,VTPtoC_l<=4} is best possible, as illustrated by the following examples.
First, we consider case \pref{RegPtoC-kappa>=4} of \Cref{RegPtoC}, and \Cref{VTPtoC_l<=4}.

\begin{eg}\label{ctrex-circulant}
For a positive integer $n\ge 3$, let $G_n$ be the graph with vertex-set $\{0,1\ldots,n-1\}$ such that two vertices $u,v\in V(G_n)$ are adjacent if and only if the difference $u-v$ is congruent to $1,2,n-2$ or $n-1$ modulo $n$ (i.e., $G_n$ is the circulant graph $\mathrm{Circ}(2n,\{ \pm 1,\pm 2\})$, see {\cite[Section 1.5]{godsil2001algebraic}}). Note that the cyclic group $\mathbb{Z}_n$ induces automorphisms $t_k\colon x\mapsto (x + k) \pmod{n}$ of $G_n$ with $k\in \mathbb{Z}_n$. In particular, each $G_n$ is vertex-transitive.

Consider the following $5$-edge path:
\[P = (0,2n-2,2n-1,1,2,3)\,.\]
The path $P$ passes through all neighbours of $0$ in $G_n$ and terminates at $3$, a non-neighbour of $0$, making it impossible for it to be contained in a cycle in $G_n$ (cf.~\cref{PathNotInCycleStrategy}).
See \cref{cyclesquare} for an illustration.

\begin{figure}[!htbp]
    \centering
    \includegraphics[width=0.3\linewidth]{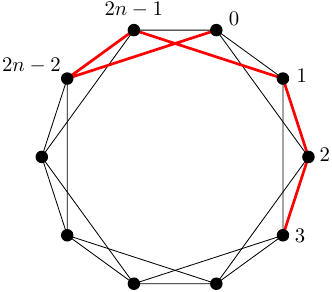}
    \caption{The graph $G_n$ and a $5$-edge path in it that does not close to a cycle}
    \label{cyclesquare}
\end{figure}

Moreover, we observe that for $\ell\geq 6$ and all sufficiently large $n$, the above path in $G_n$ can be extended with vertices $4,5,\ldots, \ell-2$ resulting in an $\ell$-edge path not closing to a cycle.
\hfill$\blacktriangle$
\end{eg}

The results of this section can be summarized as follows.

\StatePtoCVT*

\begin{proof}
Immediate from \cref{VTPtoC_l<=4} and \cref{ctrex-circulant}.
\end{proof}

\section{Paths closing to cycles in edge-transitive graphs}\label{Section:PtoCET}

In this section, we consider \cref{PtoCProb} in the context of edge-transitive graphs. 

Recall the following result.

\ETBCGMS*

It turns out that \cref{ET-3Paths} cannot be extended to neither shorter nor longer paths. For paths of length at most $2$, an obvious obstruction is formed by the star graphs $K_{1,n}$.

\begin{eg}
\label{StarsCounterExample}
Given $n\geq 2$, let $K_{1,n}$ be the star graph with $n$ leaves. 
Note that each of these graphs contains $0$-edge, $1$-edge, and $2$-edge paths. 
However, none of these paths close to a cycle, as $K_{1,n}$ is acyclic.
\hfill$\blacktriangle$
\end{eg}

To show that \cref{ET-3Paths} does not extend to longer paths, we consider the following infinite family of counterexamples.

\begin{eg}
\label{DiamondK_n}
Let $K_n$ denote the complete graph on $n\geq 3$ vertices. 
Assume for convenience that $V(K_n)=\{v_1,\ldots,v_n\}$. 
Let $K_n^{\diamond}$ be the graph obtained by replacing every edge of $K_n$ with a $4$-cycle.

More formally,
\[V(K_n^{\diamond})=\{v_1,\ldots,v_n\}\cup \bigcup_{1\leq i<j\leq n}\{v_{ij},v_{ji}\}\] 
and
\[E(K_n^{\diamond})= \bigcup_{1\leq i<j\leq n}\{\{v_i,v_{ij}\}, \{v_{ij},v_j\}, \{v_j,v_{ji}\}, \{v_{ji},v_i\}\}.\] 

Note that for $1\leq i<j\leq n$, the transposition $(v_{ij},v_{ji})$ (that is, the permutation of $V(K_n^\diamond)$ that interchanges the vertices $v_{ij}$ and $v_{ji}$ and fixes all the others) is an automorphism of $K_n^\diamond$. 
Moreover, each permutation $\alpha$ of the set $\{1,\ldots, n\}$ induces an automorphism $\overline{\alpha}$ of $K_n^\diamond$ obtained by the rule $\overline{\alpha}(v_i) = v_{\alpha(i)}$ for all $i\in \{1,\ldots, n\}$ and $\overline{\alpha}(v_{ij}) = v_{\alpha(i)\alpha(j)}$ for all distinct $i,j\in \{1,\ldots, n\}$.
The above defined automorphisms of $K_n^\diamond$ suffice to conclude that $K_n^\diamond$ is edge-transitive.

Finally, consider the following $4$-edge path in $K_n^{\diamond}$ with $n\geq 3$ arbitrary:
\[(v_{12},v_1,v_{21},v_2,v_{23})\,.\]
Note that the path passes through all of the neighbours of its endpoint $v_{12}$ and terminates at $v_{23}$, a non-neighbour of $v_{12}$.
Therefore, this path does not close to a cycle (cf.~\Cref{PathNotInCycleStrategy}); see \cref{fig:K3K4diamonds} for examples.
Furthermore, for all sufficiently large $n$, this path can be extended to a path of arbitrary length with the same property. 
\begin{figure}[!htbp]
    \centering
    \includegraphics[width=0.8\linewidth]{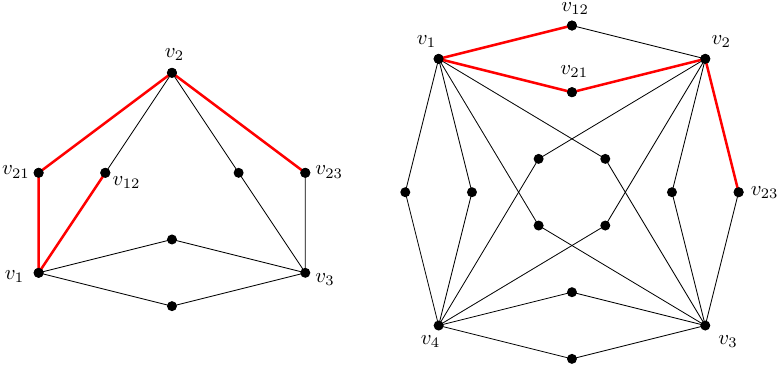}
    \caption{$K_3^\diamond$ and $K_4^\diamond$ with their respective $4$-edge paths not closing to cycles}
    \label{fig:K3K4diamonds}
\end{figure}
\hfill$\blacktriangle$
\end{eg}

\subsection{Excluding star graphs}\label{Section:ETExcludingStars}

As we have seen in \cref{StarsCounterExample}, \cref{ET-3Paths} cannot be extended to shorter paths in edge-transitive graphs due to existence of star graphs $K_{1,n}$. 
It turns out that this trivial obstruction is the only one, and in this section, we resolve \cref{PtoCProb} for the class 
of edge-transitive graphs that are not stars. 
We first make a few easy observations.

\begin{prop}
\label{PathToCycle_ETKappaTrick}
Let $G$ be a connected edge-transitive graph with minimum degree $\delta(G)$ at least~$2$.
Then, every path in $G$ of length at most $\delta(G)$ closes to a cycle.
\end{prop}
\begin{proof}
\cref{ETKappaBound} shows that $\kappa(G)=\delta(G)$. Applying \cref{PathToCycle_KappaTrick} gives the desired result.
\end{proof}

\begin{lem}\label{delta<2ImpliesStar}
Let $G$ be a connected edge-transitive graph with minimum degree at most $1$.
Then, $G$ is a star graph $K_{1,n}$ for some $n\geq 0$.
\end{lem}
\begin{proof}
Suppose first that $\delta(G)=0$.
Then, $G$ has an isolated vertex. 
Since $G$ is connected, it follows that $G$ is isomorphic to $K_1 = K_{1,0}$.

Suppose now that $\delta(G)=1$. 
Then, there exists a vertex $v\in V(G)$ with degree $1$. In particular, $v$ has a unique neighbour in $G$, call it $u$. 
If $G$ contains no vertices other than $v$ and $u$, then $G$ is isomorphic to $K_2 = K_{1,1}$.

Otherwise, let $w\in V(G)$ be an arbitrary vertex of $G$ distinct from $v$ and $u$. As $G$ is connected, it contains a path $P =(w = x_0,x_1,\ldots,x_n=u)$ from $w$ to $u$. 
Since $G$ is also edge-transitive, there exists an automorphism $\phi$ of $G$ such that $\phi(\{v,u\}) = \{w,x_1\}$. 
As automorphisms preserve the degree of a vertex, it is clear that $\phi(v)\neq x_1$, since $x_1$ is of degree at least $2$. 
Hence, $\phi(v) = w$. In particular, the degree of $w$ is equal to the degree of $v$, which is $1$. 
Therefore, we conclude that every vertex of $G$ distinct from $u$ has degree $1$. 
As $G$ is connected, it follows that $G$ is isomorphic to the star graph $K_{1,n}$ (with $u$ as its center), where $n\coloneqq |V(G)\setminus \{u\}| = |V(G)|-1$.
\end{proof}

Finally, we arrive at the following result.

\StatePtoCETnoStars*

\begin{proof}
    Since $G$ is not a star, it follows by \cref{delta<2ImpliesStar} that $\delta(G)\geq 2$. Therefore, the conclusion follows for $\ell$-edge paths with $0\leq \ell \leq 2$ by \cref{PathToCycle_ETKappaTrick} and for $3$-edge paths by \cref{ET-3Paths}. The infinitely many counterexamples for each $\ell \geq 4$ have already been discussed in \cref{DiamondK_n}.
\end{proof}

\section{Induced paths closing to induced cycles in vertex-transitive graphs}\label{Section:IndPtoCVT}

From now on, we focus on \cref{IndPtoCProb}, a variant of \cref{PtoCProb} where both the paths and the cycles under consideration need to be induced subgraphs of the given graph. 
In this section, we resolve \cref{IndPtoCProb} for the class of vertex-transitive graphs.

\begin{prop}\label{IndPToC_VT<=2}
    Let $G$ be a connected vertex-transitive graph with $|V(G)|\geq 3$. 
    Then, every induced path in $G$ of length at most $2$ closes to an induced cycle.
\end{prop}
\begin{proof}
    By \cref{VTPtoC_l<=4}, it follows that every induced $0$-edge path (i.e., vertex), as well as every induced $1$-edge path (i.e., edge) closes on a cycle.
    As the shortest cycle among all cycles in $G$ that pass through a given vertex or a given edge is necessarily induced, the conclusion follows for induced paths of length at most $1$. The fact that the same holds for induced $2$-edge paths is precisely the statement of \cref{BCGMS_VTInd2paths}.
\end{proof}

It turns out that the bound on the length of induced path in \cref{IndPToC_VT<=2} is best possible. 
In order to prove this, we first recall the following result from the literature.

\begin{lem}[{\cite[Lemma 13]{Gurvich2021Shifting}}]\label{LineGraphsLemma}
Let $G$ be a graph and let $L(G)$ be its line graph. 
Then, the following statements hold.
\begin{enumerate}[(a)]
\item\label{LineGraphsLemma-P} Let $P$ be a path of length $\ell\ge 1$ in $G$ and let $P'$ be the sequence of edges of $P$ along the path. Then $P'$ is an induced path of length $\ell-1$ in $L(G)$.
\item\label{LineGraphsLemma-C} Let $C'$ be an induced cycle of length at least four in $L(G)$. Then, the sequence of vertices of $C'$ along the cycle yields a sequence of edges of $G$ that forms a cycle $C$ in $G$.
\end{enumerate}
\end{lem}

We can now prove the following.

\begin{cor}\label{LineGraphCounterExample}
Let $\ell\ge 3$ be an integer and let $G$ be a connected edge-transitive graph containing a path $P$ of length $\ell$ that does not close to a cycle.
Let $L(G)$ be the line graph of $G$ and let $P'$ be the sequence of edges of $P$ along the path.
Then $L(G)$ is a vertex-transitive graph and $P'$ is an induced path of length $\ell-1$ in $L(G)$ that does not close to an induced cycle.
\end{cor}

\begin{proof}
We have already remarked in \cref{AutLine} that, if $G$ is edge-transitive, then $L(G)$ is vertex-transitive. 
That $P'$ is an induced $(\ell-1)$-edge path in $L(G)$ follows from \fullcref{LineGraphsLemma}{P}. Assume that $P'$ lies on an induced cycle $C'$ in $L(G)$. As $P'$ is of length $\ell - 1 \geq 2$ and induced in $L(G)$, it follows that $C'$ is of length at least four. 
By \fullcref{LineGraphsLemma}{C}, the sequence of vertices of $C'$ along the cycle yields a sequence of edges of $G$ forming a cycle $C$, which contains $P$ (since $C'$ contains $P'$). 
This is a contradiction, as it was assumed that $P$ does not close to a cycle. 
Hence, $P'$ does not close to an induced cycle in $L(G)$.
\end{proof}

\cref{LineGraphCounterExample} allows us to extend examples of edge-transitive graphs with paths not closing to cycles (such as those considered in \cref{Section:PtoCET}) to vertex-transitive graphs with induced paths not closing to induced cycles.

\begin{eg}\label{ctrex-line-diamond}
Let $K_n^{\diamond}$ denote the graph defined in \cref{DiamondK_n}, obtained by replacing every edge of the complete graph $K_n$ with a $4$-cycle. 
Recall that all members of this infinite family are edge-transitive and that for any $\ell \geq 4$ and sufficiently large $n$, $K_n^{\diamond}$ contains an $\ell$-edge path not closing to a cycle. 
By \cref{LineGraphCounterExample}, $\{L(K_n^{\diamond})\}_{n\geq 3}$ is an infinite family of vertex-transitive graphs with the property that for a given integer $\ell\geq 3$, infinitely many graphs $L(K_n^{\diamond})$ contain an induced $\ell$-edge path not closing to an induced cycle in $L(K_n^{\diamond})$.

In \cref{fig:K3diamondandlinegraph}, the graph $K_3^\diamond$ is depicted on the left, with an example of a path not closing to a cycle depicted in red. 
On the right is its line graph, with an induced path not closing to an induced cycle depicted in blue with red vertices corresponding to edges of the red path in $K_3^\diamond$.

\begin{figure}[!htbp]
    \centering
    \includegraphics[width=0.85\linewidth]{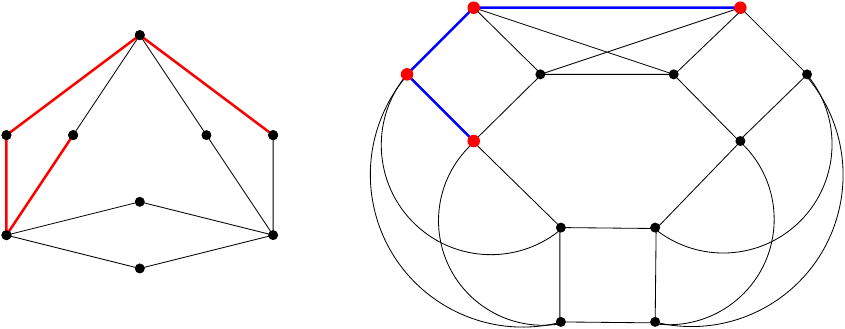}
    \caption{$K_3^\diamond$ and $L(K_3^\diamond)$ and induced paths not closing to any induced cycle}
    \label{fig:K3diamondandlinegraph}
\end{figure}
\hfill$\blacktriangle$
\end{eg}

The main result of this section can be summarized as follows.

\StateIndPtoCVT*

\begin{proof}
Immediate from \cref{IndPToC_VT<=2} and \cref{ctrex-line-diamond}.
\end{proof}

\section{Induced paths closing to induced cycles in edge-transitive graphs}\label{Section:IndPtoCET}

In this section, we first consider \cref{IndPtoCProb} for the class of edge-transitive graphs. Our starting point is \cref{IndPtoC_ET3-paths}, which we restate for convenience.

\ETindBCGMS*

Analogously to our analysis from \cref{Section:PtoCET}, we observe that star graphs (see \cref{StarsCounterExample}) are an obstruction to extending \cref{IndPtoC_ET3-paths} to induced paths of length $\ell\in \{0,1,2\}$ (note that all paths in every star graph are induced). It has already been observed in \cite{BCGMS2022} that \cref{IndPtoC_ET3-paths} does not hold for induced $4$-edge paths in edge-transitive graphs. 
We revisit their example.

\begin{eg}[{\cite[Example 1.1]{BCGMS2022}}]\label{13vCounterExample}
    The $13$-vertex graph displayed in \Cref{circulant} is edge-transitive; however, it contains an induced $4$-edge path (highlighted in red in the figure) that does not close to an induced cycle.

\begin{figure}[!htbp]
    \centering
    \includegraphics[width=0.3\linewidth]{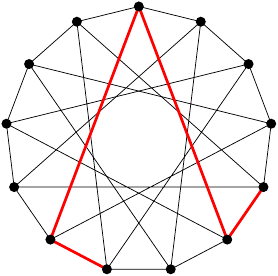}
    \caption{A $13$-vertex edge-transitive graph.}
    \label{circulant}
\end{figure}
\end{eg}

It turns out that \cref{IndPtoC_ET3-paths} cannot be extended to induced $\ell$-edge paths with $\ell \geq 5$ in edge-transitive graphs either.

\begin{eg}\label{QnCounterExample}
Let $n\geq 3$ be a positive integer and denote by $Q_n$ the $n$-dimensional hypercube graph, that is, the graph with vertex set $\{0,1\}^n$ in which two distinct vertices $x,y\in \{0,1\}^n$ are adjacent if and only if there exists a unique $i\in \{1,\ldots, n\}$ such that $x_i\neq y_i$. For $i\in \{1,\ldots, n\}$, let $e_i$ be the binary vector of length $n$ with only the $i^{\text{th}}$ coordinate non-zero.

The line graph $L(Q_n)$ of $Q_n$ is both vertex-transitive and edge-transitive, which follows from the fact that $Q_n$ is $2$-arc-transitive (see {\cite[p. 59]{godsil2001algebraic}}).

We consider the following path $P$ in $Q_n$.
\begin{equation*}
    \begin{split}
      P = (0,e_1,e_1+e_2,e_2,e_2+e_3,e_3,\ldots,e_{n-1}+e_n,e_n,e_n+e_1).
    \end{split}
\end{equation*}

The length of $P$ is $1+2\cdot(n-1)+1=2n$. Note that $P$ visits all neighbours the $0$ vertex (i.e., vertices $e_1$, $\ldots$, $e_n$) and terminates at $e_n+e_1$, a non-neighbour of $0$. By \cref{PathNotInCycleStrategy}, $P$ does not close to a cycle in $Q_n$. Hence, \cref{LineGraphCounterExample} implies that $L(Q_n)$ contains an induced $(2n-1)$-edge path $P'$, such that $P'$ does not close to an induced cycle in $L(Q_n)$. As $L(Q_n)$ is edge-transitive, by letting $n$ vary, we observe that \cref{IndPtoC_ET3-paths} is false for induced $\ell$-edge paths with $\ell \geq 5$ odd.

Note that we can augment the path $P$ with an additional edge, say $\{e_1+e_n,e_1+e_2+e_n\}$ obtaining a path $P^*$ that also does not close to a cycle in $Q_n$. 
By another application of \cref{LineGraphCounterExample}, we observe that $L(Q_n)$ also contains an induced $2n$-edge path that does not close to an induced cycle. 
Hence, by letting $n$ vary, we conclude that \cref{IndPtoC_ET3-paths} is also false for induced $\ell$-edge paths with $\ell \geq 6$ even.
\hfill$\blacktriangle$
\end{eg}

\subsection{Excluding star graphs}\label{Section:IndPtoCETExcludingStars}

As discussed previously, star graphs pose an obstruction to extending \cref{IndPtoC_ET3-paths} to shorter induced paths in edge-transitive graphs. As it turns out, quite similarly to \cref{PtoCETnoStars}, this trivial obstruction is the only one.

\begin{prop}\label{ETNotAStarInd<=2}
Let $G$ be a connected edge-transitive graph that is not a star. 
Then every induced path of length at most $3$ in $G$ closes to an induced cycle.
\end{prop}

\begin{proof}

Since $G$ is assumed not to be a star, it follows by \cref{delta<2ImpliesStar} that the minimum degree $\delta(G)$ of $G$ is at least $2$. 
In particular, every induced $0$-edge path (i.e., vertex) closes to an induced $1$-edge path (i.e., edge). 
By \cref{PtoCETnoStars}, every $1$-edge path closes to a cycle.
A cycle of minimum length through a given edge must be induced, establishing the claim for (induced) paths of length $0$ or $1$.

We now prove the claim for induced $2$-edge paths. Let $P=(x,y,z)$ be an induced $2$-edge path in $G$. Note that since $P$ is induced, it follows that $x$ and $z$ are nonadjacent.

If $G$ is vertex-transitive, the conclusion follows by \cref{BCGMS_VTInd2paths}. 
Hence, we can assume that $G$ is not vertex-transitive. \cref{edgetran-notvertrans} now tells us that $G$ is bipartite. 
Since the minimum degree of $G$ is at least $2$, the vertex $x$ has a neighbour $x' \in V(G)\setminus V(P)$. 
Note that $x'$ is not adjacent to $y$, as in this case $(x',x,y,x')$ would be a $3$-cycle in $G$, a contradiction with $G$ being bipartite. 
If $x'$ is adjacent to $z$, $(x',x,y,z,x')$ is an induced $4$-cycle containing $P$, and we are done.

The final case to consider is when $x'$ is not adjacent to $z$. 
It follows that the $3$-edge path $P'=(x',x,y,z)$ containing $P$ is an induced path in $G$. 
By \cref{IndPtoC_ET3-paths}, the path $P'$ closes to an induced cycle.
This induced cycle also contains $P$.

Finally, the fact that this claim holds for $3$-induced paths is precisely the statement of \cref{IndPtoC_ET3-paths}.
\end{proof}

The results of this subsection can be summarised as follows.

\StateIntPtoCETnoStars*

\begin{proof}
    Immediate from \cref{ETNotAStarInd<=2} and \cref{QnCounterExample}.
\end{proof}

\section{Open problems}\label{sec:open}

We conclude with a few open problems.

\cref{RegPtoC} can be reformulated as stating that the solution sets to the PtoC problem (i.e., \cref{PtoCProb}) for the class of $d$-regular graphs with connectivity at least $4$, as well as the class \hbox{of~$3$-regular} graphs with connectivity $3$, both include integers $\ell$ with $0\leq\ell\leq 4$. 
As shown by the examples depicted in \cref{fig:24,fig:Ladder}, the connectivity must be large enough for the statement to hold.

\begin{figure}[!htbp]
    \centering
    \includegraphics[width=0.4\linewidth]{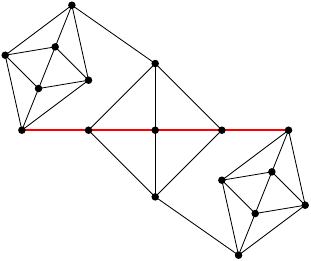}
    \caption{A $4$-regular graph with $\kappa(G)=2$ and a path of length $4$ that does not close to a cycle.}
    \label{fig:24}
\end{figure}

\begin{figure}[!htbp]
    \centering
    \includegraphics[width=0.9\linewidth]{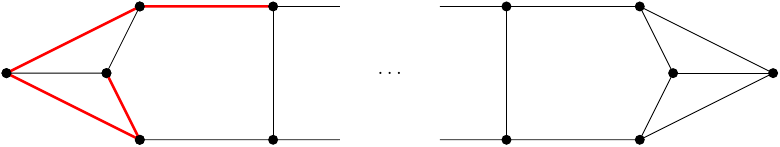}
    \caption{An infinite family of $3$-regular $2$-connected graphs each containing a path of length $4$ that does not close to a cycle.}
    \label{fig:Ladder}
\end{figure}

These observations lead to the following problem.

\begin{prob}\label{PtoCProb_Regular}
    Given a pair of nonnegative integers $(d,k)$ such that $d\geq k$, determine all integers $\ell\geq 0$ with the property that if $G$ is a $d$-regular $k$-connected graph, then every $\ell$-edge path in $G$ closes to a cycle.
\end{prob}

Note that if one defines $\mathcal{G}_{d,k}$ to be the class of $d$-regular $k$-connected graphs, then \cref{PtoCProb_Regular} is just the PtoC problem (\cref{PtoCProb}) stated for the graph class $\mathcal{G}_{d,k}$ for all values $d,k\in \mathbb{N}$ with $d\geq k$. 
Since vertex-transitive graphs are regular, solutions to \cref{PtoCProb_Regular} will generalise \cref{PtoCVT}.

One can also consider the induced version of the same problem (i.e., \cref{IndPtoCProb}) for the graph classes $\mathcal{G}_{d,k}$.

\begin{prob}\label{IndPtoCProb_Regular}
  Given a pair of nonnegative integers $(d,k)$ such that $d\geq k$, determine all integers $\ell\geq 0$ with the property that if $G$ is a $d$-regular $k$-connected graph, then every induced $\ell$-edge path in $G$ closes to an induced cycle.
\end{prob}

It might also be interesting to consider \cref{PtoCProb,IndPtoCProb} for various classes of graphs displaying combinatorial regularity not necessarily arising from group actions, for example for strongly regular graphs.

Note that in \cref{PtoCETnoStars}, given a value $\ell\geq 4$, we were only able to construct a single edge-transitive non-star graph containing an induced $\ell$-edge path not closing to an induced cycle (see \cref{13vCounterExample,QnCounterExample}). 
That is, for each non-solution $\ell\geq 4$ to the \cref{IndPtoCProb} for edge-transitive non-star graphs, we obtained a single counterexample, unlike in \cref{PtoCVT,IndPtoCVT,PtoCETnoStars}, where we constructed infinitely many counterexamples in each cases. 
This bring us to our next challenge.

\begin{prob}\label{MissingCounterExamplesProblem}
    Given an integer $\ell\geq 4$, construct an infinite family of pairwise non-isomorphic edge-transitive non-star graphs each containing an induced $\ell$-edge path not closing to an induced cycle.
\end{prob}

Note that given an $\ell\geq 4$, a solution to \cref{MissingCounterExamplesProblem} can be obtained along the lines of \cref{QnCounterExample}, in particular by using \cref{LineGraphsLemma}, if one can construct an infinite family of $2$-arc-transitive graphs each containing an $\ell$-edge path not closing to a cycle.

\subsubsection*{Acknowledgments} 
This work is supported in part by the Slovenian Research and Innovation Agency (I0-0035, research program P1-0285 and research projects J1-3003, J1-4008, J1-4084, J1-60012, and N1-0370) and by the research program CogniCom (0013103) at the University of Primorska.
Part of this research was done during the second author's visit to the University of Primorska, also supported by New Zealand Marsden Fund grant UOA-2122 and University of Auckland's DRDF.

\bibliography{References}{}
\bibliographystyle{amsplain}

\end{document}